\documentclass[reqno,a4paper,11pt]{amsart}
\usepackage{amssymb}
\usepackage[dvips, lmargin=3.5cm, rmargin=3.5cm, tmargin=4cm, bmargin=4cm]{geometry}
\usepackage[colorlinks=true, citecolor=blue, linkcolor=blue, urlcolor=blue]{hyperref}

\usepackage{xfrac}
\usepackage{graphicx}
\usepackage{hyperref}
\usepackage{amsmath}
\usepackage[T1]{fontenc}
\usepackage{charter}
\usepackage{enumitem}
\usepackage{dsfont}
\everymath{\displaystyle}
\newtheorem{thm}{Theorem}[section]
\newtheorem{defini}{Definition}[section]
\newtheorem{rem}{Remark}[section]
\newtheorem{lem}{Lemma}[section]
\newtheorem{prop}{Proposition}[section]
\newtheorem{coro}{Corollary}[section]
\newtheorem{ex}{Example}
\numberwithin{equation}{section}

\begin{document}
\title[The ergodic theory for $C_0$-semigroups and application of its local property.] {The ergodic theory for $C_0$-semigroups and application of its local property.} 
\author[]
{F. Barki \, and \, A. Tajmouati }


\maketitle
	\begin{center}
{Sidi Mohamed Ben Abdellah University}\\
	Faculty of Sciences Dhar El
	Mahraz. Fez, Morocco\\   
	Email: fatih.barki@usmba.ac.ma\,\,--\,\,abdelaziz.tajmouati@usmba.ac.ma
	
\end{center}
\vspace{0.6cm}
\quad \textbf{ Abstract}.
 Let $\{T(t)\}_{t\geqslant 0}$  be a $C_0$-semigroup of bounded linear operators on the Banach space ${X}$ into itself and  let  $A$ be their infinitesimal generator. In this paper, we show that if $T(t)$ is uniformly ergodic, then $A$ does not have the single valued extension property, which implies that $A$ must have a nonempty interior of the point spectrum.
Furthermore, we introduce the local mean ergodic for $C_0$-semigroup $T(t)$ at a vector $x\in X$ and we establish some conditions implying that $T(t)$ is a local mean ergodic at $x$.


\section{\textbf{Introduction} }

The semigroup can be used to solve a large class of problems commonly known as the Cauchy problem:
\begin{equation}\label{b}
u'(t)=Au(t), \,  t\geq 0, u(0)=u_0
\end{equation}
on  a Banach space $X$. Here $A$ is a given linear operator with domain $D(A)\subset X$ and the initial value $u_0$. The solution of (\ref{b}) will be given by $u(t)=T(t)u_0$ for an operator semigroup $\{T(t)\}_{t\geq0}$ on $X$. In this paper,  We will focus on a special class of linear
semigroups called $C_0$ semigroups which are semigroups of strongly continuous bounded linear operators. Precisely:

A one-parameter family $\{T(t)\}_{t\geqslant0}$ of bounded linear operators on  a Banach space $X$ is called strongly continuous semigroup ($C_0$-semigroup in short) \cite{Pa} if it has the following properties:
\begin{enumerate}
		\item $T(0)\, =\, I,$
	\item $T(t)T(s)\, =\, T(t+s),$
	\item The map $t\rightarrow T(t)x$ from $[0,+\infty[$ into $ {X}$ is continuous for all $x\in{X}$.
\end{enumerate}
In this case, its infinitesimal generator $A$ is defined by
$$Ax= \lim_{t\rightarrow 0^+}\frac{T(t)x-x}{t} \,  \mbox{ for all } x\in D(A),$$
with
$$D(A)=\{x\in  {X} :\,  \lim_{t\rightarrow 0^+}\frac{T(t)x-x}{t}\,   exists \, \}.$$
The Laplace transformation $\mathcal{R}(\lambda)=\int_{0}^{\infty}e^{-\lambda t}T(t)dt $
 of a $C_0$-semigroup $\{T(t)\}_{t\geqslant0}$  on $X$ is exactly the resolvent function of $A$ (see  \cite[Paragraph p.25]{Pa}), that's means  \begin{equation}
  R(\lambda, A)\, = \int_{0}^{\infty}e^{-\lambda t}T(t)dt,
 \end{equation}

\vspace{0.2cm}
	Let $\{T(t)\}_{t\geqslant0}$ be a $C_0$-semigroup of bounded linear operators in a Banach
	space $X$. Ergodic theorems \cite{Kr85} have a long tradition and are usually formulated via existence of the limits of the Ces{\`a}ro averages $C(t)$, defined as follows
\begin{equation}\label{Cesaro}
 C(t)\, :=\,  t^{-1}\int_{0}^{t}T(s)ds, \,  \mbox{ for } t\geq 0.  
\end{equation}
The $C_0$-semigroup $\{T(t)\}_{t\geqslant0}$ is said to be  mean (resp. uniformly) ergodic \cite{Kr85} if the Ces{\`a}ro averages $C(t)$ converges in the strong (resp. the norm) operator topology. This notion is completely connected to study the limit of the Abel averages $\mathcal{A}(\lambda)$ of $T(t)$, defined by
\begin{equation}
\mathcal{A}(\lambda)= \lambda \int_{0}^{\infty}e^{-\lambda t}T(t)dt , \mbox{ where } \lambda>0. 
\end{equation}
  Recall that, a $C_0$-semigroup $\{T(t)\}_{t\geqslant0}$  is called
   Abel ergodic if the limit of the Abel averages $ \mathcal{A}(\lambda)$, when $\lambda\to 0^+$, exists in the strong  operator topology.\\
  
Much of modern works has been focused on the study the connection  between   means ergodicity  and   Abel ergodicity for different class of semigroups.  E. Hille and R.S. Phillips in \cite{HP} deals with the uniform convergence of Abel averages $\mathcal{A}(\lambda)$ of semigroup of class $(A)$, a class slightly larger than $C_0$-semigroups, under the  assumption $\omega_0\leq 0$.  In \cite{Sh86}, S.Y. Shaw obtained for a locally integrated semigroup $\{T(t)\}_{t\geqslant0}$ when ${X}$ is over the complex field, under an assumption weaker than $\omega_0\leq 0$, that means $C(t)$ converges in norm if, and only if: $(i)$ the resolvent function $R(\lambda,A)$ exists for every $\lambda >0$, $(ii)$  $\|T(t)R(1,A)\|/t \to  0$ as $t\to \infty$ and $(iii)$ the Abel averages $\mathcal{A}(\lambda)$ converges in norm when $\lambda\to 0^+$. An interesting basic result needed in  this regard is  the uniform ergodic theorem of M. Lin in \cite{L2}, he treats the uniform ergodicity of a $C_0$-semigroup $\{T(t)\}_{t\geqslant0}$, under the assumption $\lim\limits_{t\to \infty}\|T(t)\|/t=0$.  
Further condition
have been obtained more recently by several authors \cite{KSZ,Su,taj}. 
	Around of the mean ergodicity,  there are many works done on different classes of semigroups defined on the Banach space ${X}$,  ~ see for instance ~ \cite{Kido,Mas,Sato, Sh-89}.\\


	This paper is organized as follows. In section 2, we give some definitions  and fundamental properties concerning the local spectral theory for a closed operator on a Banach space $ {X}$. In section 3, we shall employ local spectral theory to prove that a generator $A$ of a uniformly ergodic $C_0$-semigroup $\{T(t)\}_{t\geqslant0}$ does not have SVEP (Single Valued Extension Property),  which implies that $ ~ int\big(\sigma_p(A)\big) ~ \neq ~ \emptyset $. Furthermore, we study some conditions implying that a $C_0$-semigroup $\{T(t)\}_{t\geqslant0}$ locally satisfies the convergence of Ces{\`a}ro averages $C(t)$,    which means that $ T(t) $ is local mean ergodic at some $x\in X$. More precisely, we prove that if a generator $A$ of $C_0$-semigroup $\{T(t)\}_{t\geqslant0}$ has the SVEP and $0$ is a pole of the local resolvent function $\hat{x}_A$ of $A$ at $x$, then $T(t)$ is local mean ergodic at $x$.  

%
\section{\textbf{Preliminaries}} 

Throughout this paper, $\mathcal{B}(X)$ denotes the Banach algebra of all bounded linear operators on a complex Banach space ${X}$ into itself. Let $A$ be a closed linear operator on $ {X}$ with domain $D(A)\subset  {X}$, we denote by $N(A)$, $R(A)$, $\sigma(A)$, $\rho(A)$, $\sigma_{p}(A)$, $\sigma_{su}(A)$ and $R(.,A)$, the kernel, the range, the spectrum, the resolvent set, the point spectrum, the surjectivity spectrum and the resolvent operator of $A$, respectively. \vspace{0.2cm}

Recall that for  a closed linear operator $A$ and $x\in X$, the local resolvent of $A$ at $x$,  $\rho_{A}(x)$ defined as the union  of all  open subset $U$ of $\mathbb{C}$ for which there is an analytic function $f: U\rightarrow D(A)$ such that the equation $(A-\mu I)f(\mu)=x$ holds for
all $ \mu \in U$. The local spectrum $\sigma_A(x)$ of $A$
at $x$ is defined as $\sigma_A(x)=\mathbb{C}\setminus \rho_A(x)$.
Evidently, $\sigma_A(x)\subseteq\sigma_{su}(A)\subseteq\sigma(A)$, \,  $\rho_A(x)$ is open and $\sigma_A(x)$ is closed. 
\begin{lem}\cite{E, Mul}\label{l0}
	 Let $A$ be a closed  linear operator on a complex Banach space $X$. Then, 
	 $\mu\in\rho_{A}(x)$ if and only if there exists a sequence $(x_i)_{i\geq0}\subseteq D(A)$,  such that $(A-\mu)x_0=x$,  $(A-\mu)x_{i+1}=x_i$ and $\sup_{i\geq1}||x_i||^{\frac{1}{i}}<\infty$.
\end{lem}

Next, let $A$ be a closed  linear operator on $X$, $A$ is said to have the single
valued extension property at $\lambda_{0}\in\mathbb{C}$ (SVEP  for brevity \cite{Mb}) if
for every open disk $D_{\lambda_0}\subseteq\mathbb{C}$ centred at
$\lambda_{0}$, the only  analytic function  $f: D_{\lambda_0}\rightarrow D(A)$ which satisfies
the equation $(zI-A)f(z)=0$ for all $z\in D_{\lambda_0}$, is the function $f\equiv 0$. Moreover, $A$ is said to have the SVEP if $A$ has the SVEP at every $\lambda\in\mathbb{C}$. Denote by \begin{equation}\label{S}
S(A)=\{\lambda\in \mathbb{C}: A\mbox{ has not  the SVEP at } \lambda\}.
\end{equation} 
Clearly, $A$ has the SVEP if and only if $S(A)=\emptyset$. Furthermore,  we have the following characterization: $\mu\in S(A)$ if and only if there exists a sequence $(x_i)_{i\geq0}\subseteq D(A)$ not all of them equal to zero such that $x_0=0$,
$\, (A-\mu)x_{i+1}=x_i$ and  $\sup_{i\geq1}||x_i||^{\frac{1}{i}}<\infty$. From definition, we can see that every operator without SVEP has a nonempty open set consisting of eigenvalues. So, if the point spectrum of an operator has empty interior, then the operator has automatically SVEP. If $A$  has the SVEP,
then for every $x\in X$ there exists a unique maximal analytic function $\hat{x}_A$, called
the local resolvent function of $A$ at $x$, defined on $\rho_A(x)$, which satisfies $(\lambda I -A)\hat{x}_A(\lambda)=x$, everywhere.
Moreover, if $A$ has the SVEP, then we have the following equivalence:
$$ \sigma_A(x) =\emptyset \, \Longleftrightarrow \, x=0.$$  
It is worth noticing that according to \cite[Remark 2.4 (d)]{PA} the following implication hold:
$$ int\big(\sigma_{p}(A)\big)=\emptyset  \, \implies A \mbox{ has the SVEP }.$$

The local spectral radius $r_A(x)$ of $A$ at $x\in X$, is defined by 
$$ r_A(x):= max \{ |z|: z\in \sigma_A(x)\}.$$
If $A$ has the SVEP, then the following
local spectral radius formula hold: 
 $$r_A(x) =\limsup_{n\to \infty}\|A^n(x)\|^{\frac{1}{n}}.$$ 
A complete study of basic notions of local spectral theory can be found in \cite{PA,E,Mb}.\\ 

Let $A$ be a closed linear operator with domain $D(A)\subset X$, the analytic core $K(A)$ and the quasi nilpotent part $H_0(A)$ of $A$, are defined respectively by
\begin{center}
	$K(A) := \big \{x \in X : \,  \exists (x_n)_{n\geq0} \, \subset  \, D(A) \, \mbox{ and } \, \delta > 0 \, \mbox{ such that } x_{0}=x, \newline A x_n = x_{n-1} \,\,\,\,\forall n \geq1 \mbox{ and } \|x_{n}\| \leq \delta^{n}\|x\| \big\}.$
\end{center}
\begin{center}
	$H_0(A) := \big \{x \in D^{\infty}(A) : \,   \lim\limits_{n\to \infty }\|A^nx\|^{\frac{1}{n}}  =0   \big\},$ \, with  $D^{\infty}(A)= \cap_{n\geq1} D(A^n)$.
\end{center}
Both spaces were thoroughly studied by \textsc{M. Mbekhta} in \cite{Mb1} and \cite{Mb}. The next proposition gives  a local spectral characterization for the analytic core $K(A)$ and the quasi nilpotent part $H_0(A)$ of $A$.
\begin{prop}\cite[ Proposition 1.3]{Mb}
	Let $A$ be a closed linear operator on the Banach space $X$. Then, we have the following:
\begin{enumerate}
	\item $K(A)=\big\{x\in X: \, 0\in \rho_{A}(x)\big \}.$
	\item $H_0(A) \subseteq\big\{x\in X: \, \sigma_{A}(x) \subseteq \{0\} \big\} $\,  $($we get the equality when $A$ has the SVEP$)$.
\end{enumerate}
\end{prop}

\vspace{0.2cm}

We recall a characterization of the poles of the local resolvent function obtained in the monograph \cite{Bur}. 
\begin{prop}\label{p1}\cite[Proposition 3.1]{Bur} Assume that $T \in \mathcal{B}( X )$ has the SVEP and let $x \in X$. Then, $\sigma_T \big((\alpha-T)^nx\big)\neq \sigma_T(x)$ if and only if $\alpha$ is a pole of the local resolvent function $\hat{x}_T$ of order less than or equal to $n$.
\end{prop} 
\vspace{0.2cm}

 The following results have been proven by \textsc{T. Berm{\'u}dez} and al in \cite{Bur}. In fact, we can easily check that these results are also satisfactory when we choose $A$ is a closed linear operator with domain  $D(A) \subsetneq X$.

 \begin{coro}\label{c1}\cite[Corollary 3.2]{Bur}
 Assume that $T \in \mathcal{B}(X)$ has the SVEP and let $x \in X$ and let $\hat{x}_T$ be the local resolvent function of $T$ at $x$. Then, the following assertions hold:
 \begin{enumerate}
 	\item If $\hat{x}_T$ has a pole of finite order at $\alpha$, then $\alpha \in \sigma_p(T)$.
 	\item The local resolvent function $\hat{x}_T$ at $x$ has a pole of order $n$ at $\alpha$ if and only if $\alpha\in \sigma_T\big((\alpha-T)^{n-1}x\big)\backslash \sigma_T\big((\alpha-T)^{n}x\big)$.
 	\item If $\lambda \in \rho_T(x)$ and $y=\hat{x}_T(\lambda)$, then $\hat{x}_T$ has a pole of order $n$ at $\alpha$ if and only if $\hat{y}_T$ dose. 
 \end{enumerate}
 \end{coro}
\begin{lem}\label{lem}\cite[Theorem 3.3]{Bur}
	Assume that $T \in \mathcal{B}(X)$ has the SVEP and let $x \in X$ and let $\hat{x}_T$ be the local resolvent function of $T$ at $x$. Then, the following assertions are equivalent: 

	$(i)$\,  $\alpha $ is a pole of $\hat{x}_T$ of order $n$.
		
	$(ii)$\, There exists a unique decomposition \, $x=y+z$, \,  such that  $y$  belongs to \, $N(\alpha-T)^n\backslash N(\alpha-T)^{n-1}$ \, and $\sigma_T(z)=\sigma_T(x)\backslash \{\alpha\}$. 	
\end{lem} 

\begin{rem}
	Let $T\in \mathcal{B}(X)$ has the SVEP and $\lambda\in \mathds{C}$. It is well known $($see \cite{Bur}$)$ that if $\lambda$ is a pole of local resolvent function $\hat{x}_T$ for all $x\in X$ does not imply that $\lambda$ is a pole of the resolvent function of $T$. But, if we suppose further that $\{\lambda\} \subsetneq \sigma_{A}(x)$, the implication becomes true.
\end{rem}


Let $\{T(t)\}_{t\geqslant0}$ be a $C_0$-semigroup with infinitesimal generator $A$. It is known  that the $C_0$-semigroup is uniquely determined by its generator $A$, and we have the following properties (see \cite[Theorem 2.4]{Pa}):  
	\begin{enumerate}
\item  $A$ is closed and $\overline{D(A)}= {X}$.
\item  For all $x\in {X}$ and $t\geq 0$,
\begin{equation}\label{equ1}
\int_0^tT(s)xds\in D(A) \,\,\mbox{ and }\,\,  A\int_0^tT(s)xds=T(t)x-x.
\end{equation}
\item  For all $x\in D(A)$  and $t\geq 0$,

\begin{equation}\label{equ2} 
T(t)x\in D(A) \,  \mbox{ and } \,  T'(t)=AT(t)x=T(t)Ax. 
\end{equation}
\item For each $\lambda \in \rho(A)$,

 $(i)$ \,  $(\lambda I - A)R(\lambda,A)x = x$, \, for every $x \in {X}$.  

$(ii)$ \, $ R(\lambda,A) (\lambda I- A)x = x $, \, for every $x \in D(A)$.
\end{enumerate}
\begin{lem}\label{l1} \cite[Lemma 5.2]{Su} \label{ll}	Let $\{T(t)\}_{t\geqslant0}$ be a $C_0$-semigroup  of bounded linear operators  on a Banach space $ {X}$, and  let $ A$ be their infinitesimal generator.  Then, the following relations hold:
	
	$(1)$\,\, $R(A) = \big(\lambda R(\lambda,A)-I\big){X}$.
	
	$(2)$\,\, $ N(A) = \{x \in {X} :\,\,  \lambda R(\lambda,A)x = x \}=  {fix}\{T(t)\}$, \\
	where 	$ {fix}\{T(t)\} = \{x\in {X}:\,\,  T(t)x = x;\,\,\forall t \geq 0\}$ is the set of the fixed points of $T(t)$.
\end{lem}

%
 \section{\textbf{Main results}}
 
In terms of the ergodic decomposition, we start the present section by the following lemma which we need in the sequel.

 \begin{lem}\label{l2}
 	Let $\{T(t)\}_{t\geqslant0}$ be a $C_0$-semigroup  of bounded linear operators on ${X}$ such that $\lim\limits_{t\to \infty}\frac{\|T(t)x\|}{t}=0$  for all $x\in  {X}$, and let $A$ be the infinitesimal generator of $T(t)$. 
 	If $y\in R(A)$ and $z\in N(A)$, then $$ \frac{1}{t}\int_{0}^{t}T(s)(y+z)ds =z+O\Big(\frac{1}{t}\Big),\, \mbox{ as } t\to \infty.$$
 	
 \end{lem}
 
 \begin{proof} Let $A$ be the infinitesimal generator of a $C_0$-semigroup $\{T(t)\}_{t\geqslant0}$, and let  $C(t)$ be the Ces{\`a}ro averages define on (\ref{Cesaro}). \\
 	Let $z\in N(A)$. Then, according to the second statement of Lemma \ref{l1},  $T(t)z=z$. Hence, 
 	it follows that $$C(t)z = \frac{1}{t}\int_{0}^{t}T(s)zds = z.$$
	Let $y\in R(A)$. Then, there exist $w$ belongs to domain $ D(A)$, such that $y=Aw$. Hence, from the identity (\ref{equ2}), we obtain 
 	\begin{eqnarray*}
 	C(t)y &=&
 		\frac{1}{t}\int_{0}^{t}T(s)yds \\
 		&=& \frac{1}{t}\int_{0}^{t}T(s)Awds\\
 		&=& \frac{1}{t}\int_{0}^{t} \frac{d}{ds}[T(s)w]ds\\
 		&=& \frac{1}{t}[T(t)w-w].
 	\end{eqnarray*}
 	Therefore, 
 	\begin{eqnarray*}
 		\Big\|\frac{1}{t}\int_{0}^{t}T(s)yds\Big\| &\leq& \Big\|\frac{T(t)w}{t}\Big\| + \Big\|\dfrac{w}{t}\Big\|\\
 		&\leq& C^{ste} \, \Big\|\frac{w}{t}\Big\|.
 	\end{eqnarray*}
 	Then, $$ \frac{1}{t}\int_{0}^{t}T(s)(y+z)ds=z+O\Big(\frac{1}{t}\Big),\, \mbox{ as } t\to \infty.$$ 		
 \end{proof}
 
 \begin{defini}
 Let $\{T(t)\}_{t\geqslant0}$ be a $C_0$-semigroup of bounded linear operators on $ {X}$ and let $x\in X$. We say that $T(t)$ is a local mean ergodic at $x$ if the limit of the Ces{\`a}ro averages $C(t)x= \frac{1}{t}\int_{0}^{t}T(s)xds$ exists when $t \to \infty$. 

 \end{defini}
 
 \begin{rem}
 	\begin{enumerate}
 		\item A $C_0$-semigroup $\{T(t)\}_{t\geqslant0}$ is mean ergodic  if and only if $T(t)$ is a local mean ergodic at $x$ for every $x\in X$.
 		\item If $\{T(t)\}_{t\geqslant0}$ be a contraction semigroup of bounded linear operators on  $X$ and $x\in  N(A)\oplus \overline{R(A)}$, then $T(t)$ is a local mean ergodic at $x$.
 		\item If there exists $x\in X$ such that $T(t)x \longrightarrow 0$ as $t\to \infty$, then $T(t)$ is a local mean ergodic at $x$.  
 	\end{enumerate}
 \end{rem}

	 Let  $\{T(t)\}_{t\geqslant0}$ be a $C_0$-semigroup  of bounded linear operators  on a Banach space ${X}$, we consider the following subsets of $X$:
 
 \begin{equation}
  X_{0}=\big\{ x\in X: \, \lim\limits_{t\to \infty} \frac{\|T(t)x\|}{t} =0 \big\}.
 \end{equation}
 \begin{equation} X_{me}=\big\{ x\in X: \, C(t)x \mbox { converges } \big\}.
  \end{equation}
 Clearly, these two sets are invariant under the operators $T(t)$ for all $t\geq 0$, but
 not necessarily closed.
  
\begin{prop}\label{p3}
	Let $\{T(t)\}_{t\geqslant0}$ be a $C_0$-semigroup  of bounded linear operators  on a Banach space ${X}$, with $A$ be their infinitesimal generator. If $\lim\limits_{t\to \infty}{\|T(t)x\|}/{t} =0$ for all $x\in X$,  then 
	 $$ X_{me}=\big\{ x\in X: \,  \lim\limits_{t\to \infty}C(t)x=0\,  \big\} \oplus N(A).$$
\end{prop}

\begin{proof}
		From Lemma \ref{l1}, we can easily see that 
		$$\big\{ x\in X: \,  \lim\limits_{t\to \infty}C(t)x=0\,  \big\} \cap N(A) =\{0\}.$$
		Clearly, we have  $ X_{me} \supset \big\{ x\in X: \,  \lim\limits_{t\to \infty}C(t)x=0\,  \big\} \oplus N(A)$.\\
		Now, let  $x\in   X_{me}$, \,  then there exists an operator $P\in \mathcal{B}(X)$ such that $$\lim\limits_{t\to \infty}\|C(t)x
		-Px\|\to  0 \, \mbox{ as } \, t\to \infty.$$ It known by the mean ergodic theorem \cite{Sato} that the limit $P$ is the projection of $X$ onto $N(A)$ along $\overline{R(A)}$, which is equivalent that $I-P$ is the projection of $X$ onto $\overline{R(A)}$ along $N(A)$. So, we can write $x=(x-Px)+Px$ and we can infer, from the above, that  $(I-P)x \in \big\{ x\in X: \,  \lim\limits_{t\to \infty}C(t)x=0\,  \big\} $ and $Px \in
		N(A)$.\\ Therefore,  \,  $ X_{me} \subset \big\{ x\in X: \,  \lim\limits_{t\to \infty}C(t)x=0\,  \big\} \oplus N(A)$, and the equality hold.
\end{proof}

 Let $\{T(t)\}_{t\geqslant0}$  be a $C_0$-semigroup of bounded linear operators on a Banach space ${X}$. We use  $\omega_0(x)$ to denote the local growth bound of $\{T(t)\}_{t\geqslant0}$ at $x\in X$, defined by 
 \begin{equation}\omega_0(x)= \inf\{\omega \in \mathds{R}: \|T(t)x\|\leq Me^{\omega  t} \,; \mbox{ for } M>0, \forall t\geq 0 \}.
 \end{equation}  \vspace{0.1cm}
     
  Now, we obtain some local results using the local spectral radius formula of $T(t)$, for all $t\geq 0$, and the local growth bound $\omega_0(x)$ of $T(t)$ at $x\in X$.  
 
 \begin{thm}\label{t3}
 	Let $\{T(t)\}_{t\geqslant0}$  be a $C_0$-semigroup of bounded linear operators on a Banach space ${X}$  with $A$ be their infinitesimal generator and let $x\in X$. 
 	If $A$ has the SVEP, then $r_{T(t)}(x)=e^{t\omega_0(x)}$, for all $t\geq 0$.
 	Moreover, if $x\in X_0$ then $x\in H_0(A)$.
 \end{thm}
   \vspace{0.1cm}
 
 We need the following auxiliary result to prove this theorem.
 
 \begin{lem}\label{l3}
 	Let $\{T(t)\}_{t\geqslant0}$ be a $C_0$-semigroup of bounded linear operators on a Banach space ${X}$  with $A$ be their infinitesimal generator and let $x\in X$. If $x\in X_0$ then $\omega_0(x)\leq 0 $. 
 \end{lem}
   \vspace{0.1cm}
    
 \begin{proof}
 	Let  $\{T(t)\}_{t\geqslant0}$ be a $C_0$-semigroup of bounded linear operators on a complex Banach space ${X}$ with $A$ be their infinitesimal generator. Let $x\in X$ such that $\lim\limits_{t\to \infty} \frac{\|T(t)x\|}{t} =0$. Then,
 	there exists a large $t$ and $\epsilon>0$ such that $\|T(t)x\|\leq \epsilon t$. 
 	First, let show that $\omega_0(x)=\lim_{t\to \infty}\frac{1}{t}log\|T(t)x\|$. \\
 	Indeed, we use the function 
 	$$t\to \log\|T(t)\|.$$
 	It is clear that, the function is subadditive. Then,
 	$$ \inf_{t\geq 0}\frac{1}{t}log\|T(t)x\|= \lim_{t\to \infty}\frac{1}{t}log\|T(t)x\|.$$
 	Next, we can define $\omega= \inf_{t\geq 0}\frac{1}{t}log\|T(t)x\|$, it follows that:
 	$$ e^{\omega t} \leq \|T(t)x\| \, \mbox{ for all } t\geq 0.$$
 	Then, from the definition of $\omega_0(x)$, we get $\omega\leq \omega_0(x)$.\\
 	Now, let choose $\beta>\omega$. Then, there exists $t_0>0$ such that 
 	$$\frac{1}{t}log\|T(t)x\| < \beta\, \mbox{ for all } t\geq t_0.$$
 	Hence, $$\|T(t)x\| < e^{\beta t}\, \mbox{ for all } t\geq t_0,$$
 	and at  $[0,t_0]$, the norm of $T(t)$ remains bounded, then by the uniform boundedness principle, we find $M\geq 1$ such that $$ 	 \|T(t)x\| \leq M e^{\beta t}\, \mbox{ for all } t\geq 0.$$ 
 	Therefore $\omega_0(x)\leq \beta$. Since we have already prove that $\omega \leq \omega_0(x)$, this implies that $\omega_0(x)= \omega $. 
 	Finally, we get $\omega_0(x)= \lim_{t\to \infty}\frac{1}{t}log\|T(t)x\|$. \\ As mentioned above, that  for a large $t$, we have $\|T(t)x\|\leq \epsilon t$. Then,
 	\begin{eqnarray*}
 		\omega_0(x)  &\leq& \lim_{t\to \infty}\frac{1}{t}log(\epsilon t) \\
 		&\leq&   \lim_{t\to \infty}\frac{1}{t} \big( log(\epsilon) + log(t)\big).  
 	\end{eqnarray*}
 	Consequently, we obtain  \, $\omega_0(x)\leq 0$. 
 \end{proof} 
 
 \begin{proof} of Theorem \ref{t3}:
 	Let  $\{T(t)\}_{t\geqslant0}$ be a $C_0$-semigroup of bounded linear operators on a Banach space ${X}$, with $A$ be their infinitesimal generator. Let $x\in X$ and assume that $A$ has the SVEP. Then, to applied the local radius formula to the operators $T(t)$ for all $t\geq 0$ at $x\in X$, it is necessary to check that $T(t)$ has the SVEP for all $t\geq 0$.
 	Let us first verify that  $S(T(t))\subset e^{tS(A)}$, where $S(A)$ defines in $(\ref{S})$. To show this, and without loss of generality, it suffices to prove that
 	$e^{\lambda t}-T(t)$ has not  SVEP at $0$, then $\lambda -A$ has not SVEP at $0$.  Indeed,
 	let $\lambda\in \mathds{C}$ such that
 	 $e^{\lambda t}-T(t)$ has not  SVEP at $0$, then  there exists $(x_i)_{i\geq 0}\in X$ such that: $$x_0=0,\, \big(e^{\lambda t}-T(t)\big)x_i=x_{i-1} \, \mbox{ and }\sup_{i}\|x_i\|^{\frac{1}{i}}<\infty.$$
 	Now, we use the following identity \cite[Lemma 2.2, p 45]{Pa}:
 	$$ \big(e^{\lambda  t}- T(t)\big)x = (\lambda -A)B_{\lambda}(t)x  \, \mbox{ for all } t\geq 0,$$
 	where $ \,  B_{\lambda}(t)x\, =\, \int_{0}^{t}e^{\lambda (t-s)}T(s)xds$, \, which is belongs to the domain $D(A)$ of $A$. So,
 	let's denote $y_i=B_{\lambda}^{i}(t)x_i$, then $(y_i)_{i\geq 0}\subseteq  D(A)$ and $y_0=x_0=0$.
 	Therefore
 	\begin{eqnarray*}
 		(\lambda-A)y_i &=& (\lambda-A)B_{\lambda}(t) B_{\lambda}^{i-1}(t)x_i \\
 		&=& (e^{\lambda t}-T(t))B_{\lambda}^{i-1}(t)x_i \\
 		&=& B_{\lambda}^{i-1}(t)(e^{\lambda t}-T(t))x_i \\
 		&=& B_{\lambda}^{i-1}(t)x_{i-1}\\
 		&=& y_{i-1}.
 	\end{eqnarray*}
 	On the other hand, we have  \begin{equation}\label{eqsup}
 	\|y_i\|=\| B_{\lambda}^{i}(t)x_i\|\leq\|B_{\lambda}^{i}(t)\| \|x_i\|\leq M^i\|x_i\|.
 	\end{equation}
 	Then, there exists  $M>0$ such that $$\sup_{i}\|y_i\|^{\frac{1}{i}}\leq\sup_{i} M\|x_i\|^{\frac{1}{i}}<\infty.$$
 	Therefore,  $\lambda -A$ has not SVEP at $0$.
 	 Hence,  $$S(T(t))\subset e^{tS(A)}.$$ 
 After that, we conclude that
 	    $T(t)$ has the SVEP for all $t\geq 0$ if $A$ has the SVEP.\\ 
 	  Consequently,
 	$$r_{T(t)}(x)= \lim_{n\to \infty}\|T(t)^nx\|^{\frac{1}{n}} \, \mbox{ for all } t\geq 0.$$  
 	It follows that 
 	\begin{eqnarray*}
 		r_{T(t)}(x) & = &  \lim_{n\to \infty} e^{ t \frac{1}{nt} log \|T(nt)x\|} \\
 		& = & e^{ t  \lim\limits_{n\to \infty} \frac{1}{nt} log \|T(nt)x\|} \\
 		& = & e^{ t \omega_0(x)}. 
 	\end{eqnarray*}  
 
 	Now we prove that $x\in H_0(A)$ as follows. From Lemma \ref{l3} and the previous result, we obtain that,   if  $x\in X_0$,  then $r_{T(t)}(x) \leq 1$ for all $t\geq 0$.  Therefore,  $\sigma_{T(t)} (x) \subset \overline{\mathds{D}(0,1)}$  for all $t\geq 0$,  where $\mathds{D}(0,1)$ is the unit disk.
 	 Let's show that  $e^{t\sigma_{A}(x)}\subseteq \sigma_{T(t)}(x) \mbox{ for all }  t\geq 0$.  Indeed, let $\lambda \in \mathds{C}$ such that $e^{\lambda t}\notin \sigma_{T(t)}(x)$ for all $ t\geq 0$, it follows from Lemma \ref{l0} that there exists a sequence $(x_i)_{i\geq 0}\subseteq X, $
 	such that $$\big(e^{\lambda t}-T(t)\big)x_0=x, \,\,  \big(e^{\lambda t}-T(t)\big)x_i=x_{i-1} \mbox{  and  } \sup_{i\geq 1}\|x_i\|^{\frac{1}{i}}<\infty.$$
 	Let $y_i=B_{\lambda}^{i+1}(t)x_i$, with $  B_{\lambda}(t)x=\int_{0}^{t}e^{\lambda (t-s)}T(s)xds$, 
 	 then $(y_i)_{i\geq 0}\subseteq  D(A)$ and   $y_0=B_{\lambda}(t)x_0$. Moreover, we use the technique employed above, we obtain that : $
 		(\lambda-A)y_i = y_{i-1}  $ for all $i\geq 1$
 	 and the inequality  $(\ref{eqsup})$ implies that  ${\small \sup_{i\geq1}\|y_i\|^{\frac{1}{i}} <  \infty}$.  
 	Therefore, $\lambda \notin \sigma_A(x)$ by Lemma \ref{l0}, hence $e^{t\sigma_{A}(x)}\subseteq \sigma_{T(t)}(x),\,   \mbox{ for all } t\geq 0.$
 	Since  $\sigma_{T(t)} (x) \subset \overline{\mathds{D}(0,1)}$ for all $t\geq 0$ and by the above inclusion,  we infer that  $\sigma_A(x)\subseteq \{0\}$,   which yields that $x\in H_0(A)$. Moreover, Since $A$ has the SVEP then $\sigma_A(x) = \{0\}$.
 \end{proof}

 \begin{coro}\label{coro1}
	Let $\{T(t)\}_{t\geqslant0}$ be a $C_0$-semigroup  of bounded linear operators  on a   Banach space ${X}$ such that $\lim\limits_{t\to \infty}{\|T(t)x\|}/{t} =0$ for all $x\in X$. Let $A$ be the infinitesimal generator of $T(t)$. If $A$ has the SVEP, then $A$ is bounded and $\sigma(A)=\{0\}$. 
\end{coro}

 \begin{proof}By Theorem \ref{t3}, we have $\lim\limits_{t\to \infty}{\|T(t)x\|}/{t} =0$ satisfies for every $x\in X$. Then, $H_0(A)=X$ and by the definition of $H_0(A)$, we get $X=D^{\infty}(A)$. Hence, $X=D(A)$ which implies that $A$ is bounded. Furthermore, 
 	$\sigma(A)=\{0\}$  comes immediately by $A$ has the SVEP and $\sigma_{A}(x)=\{0\}$ for every non-zero $x\in X$. 
 	\end{proof}
   \vspace{0.1cm}
   
  Our main result will be the following theorem, which we will derive from the previous Corollary and Theorem \ref{t3}. 
\begin{thm}\label{thm2} Let $\{T(t)\}_{t\geqslant0}$ be a $C_0$-semigroup of bounded linear operators on a  Banach space ${X}$  with $A$ be their infinitesimal generator such that $\lim\limits_{t\to \infty}{\|T(t)x\|}/{t} =0$ for all $x\in X$. If $T(t)$ is uniformly ergodic, then {\small$int\big(\sigma_{p}(A)\big)\neq\emptyset$}. 
\end{thm}

  \begin{proof} Let $A\neq 0$ be the infinitesimal generator of a $C_0$-semigroup $\{T(t)\}_{t\geqslant0}\subset  \mathcal{B}(X)$.  By the uniform ergodic theorem \, \cite[ Theorem]{L2},\,  $T(t)$ is uniformly ergodic then $ \, \, X=R(A)\oplus N(A)$ with $R(A)$ is closed. Hence, for every non-zero vector $x\in X$ there exists a unique decomposition $x=y+z$ such that $y\in N(A)$ and $z\in R(A)$. Clearly, we have  $\sigma_{A}(y)=\{0\}$ and $0\in \rho_{A}(z)$.\\
  Now, we suppose that $A$ has the SVEP, then we get \, $\sigma_A(x)= \sigma_{A}(y) \cup \sigma_{A}(z)$.  
  	 It follows from Corollary \ref{coro1} that $A$ is bounded and $\sigma_{A}(x)=\{0\}$ for every non-zero $x\in X$. Then we deduce that $  \sigma_{A}(z)=\emptyset $, which implies that $z=0$.  
  	 Therefore $x\in N(A)$, hence $X=N(A)$ which means that A is  a zero operator. Absurd, then $A$ does not have the SVEP, which yields that $int\big(\sigma_{p}(A)\big) ~ \neq ~ \emptyset$.
 \end{proof}
 
  \vspace{0.1cm}
 
  Now, we study some conditions under which the $C_0$-semigroup $\{T(t)\}_{t\geqslant 0}$ satisfies the convergence of Ces{\`a}ro averages $C(t)$ locally, i.e. $T(t)$ is local mean ergodic at some vector $x\in X$.
 \begin{thm}\label{t2}
 		Let $\{T(t)\}_{t\geqslant 0}$ be a $C_0$-semigroup  of bounded linear operators  on ${X}$, with $A$ be their infinitesimal generator and let  $x\in X_0$. If one of these assertions hold:
 		\begin{enumerate}
 			\item \label{1}  $x\in N(A)\oplus R(A)$,  
 			\item    $x\in K(A)$, which is equivalent to $0\in \rho_A(x)$,  
 			\item  $\lambda R(\lambda,A)x $ converges as $\lambda \to 0^+$, 
 			\item $A$ has the SVEP and $0$ is a pole of the local resolvent function $\hat{x}_A$ of order $\leq1$.
 		\end{enumerate}
 		Then $T(t)$ is a local mean ergodic at $x$.
 \end{thm}

  \begin{proof}
 Let $\{T(t)\}_{t\geqslant 0}$ be a $C_0$-semigroup  of bounded linear operators  on a Banach space ${X}$, with $A$ be their  generator. Let $x\in X$ such that   $\lim\limits_{t\to \infty} \frac{\|T(t)x\|}{t} =0$.  
 \begin{enumerate}
 	\item Obviously comes from the Lemma \ref{l2}. 
 	\item It is known that $K(A)$ is a subspace of $R(A)$. Therefore, if $x\in K(A)$, then $T(t)$ is a local mean ergodic at $x$, with the limit equal to 0. 
 	\item If $\lambda R(\lambda,A)x $ converges as $\lambda \to 0^+$, then it follows from \cite[Theorem 1]{taj} that $x\in N(A)\oplus R(A)$. Therefore, $T(t)$ is a local mean ergodic at $x$ by (\ref{1}). 
 	\item Suppose that $A$ has the SVEP, then by Lemma \ref{lem} we have: $0$ is a pole of the local resolvent function $\hat{x}_A$ of order  less than or equal to $1$, implies that there exists a unique decomposition $x=y +z$ such that $y \in N(A)\backslash\{0\}$ and $ \sigma_A(z)= \sigma_A(x)\backslash\{0\}$. It follows that $ \sigma_A(y)=\{0\}$ and $z\in R(A)$, which yields that $x\in N(A)\oplus R(A)$. 
 	Therefore, $T(t)$ is a local mean ergodic at $x$. 
 \end{enumerate}
\end{proof} 
For $\{T(t)\}_{t\geqslant 0}$ be a $C_0$-semigroup  of bounded linear operators  on  ${X}$ and $x\in X$, we observe that if $\{T(t)\}_{t\geqslant 0}$ is a local mean ergodic at $x$, does not imply that $0$ is a pole of the local resolvent function $\hat{x}_A$, as shown in the following example. 
 \begin{ex}
 	Let $A$ be the right shift operator on the Hilbert space $ l_2(\mathds{N})$, defined by $$A(e_1)=0  \mbox{ and }  A(e_{n+1})= \frac{e_{n+2}}{n+1}.$$ 
 	We associate $ A $ with the family $T(t)=e^{tA}$, which means that $T(t)$ is a uniformly bounded semigroup. Clearly, \,  $\lim\limits_{t\to \infty} \frac{\|T(t)x\|}{t} =0$  for all $x\in l_2(\mathds{N})$ and $\sigma(A)=\{0\}$, hence  $int(\sigma(A))=\emptyset$, which implies that $A$ satisfies the SVEP. Consequently, we have $\sigma_A(x) =\{0\}$ for  every non-zero $x\in  l_2(\mathds{N})$.
 	Taking $x=y + z$,  where $y =e_1$  and $z=e_3$. Its easy to see that $y\in N(A)$ and $z\in R(A)$ with $\sigma_{A}(z)=\{0\}$. Hence, by Theorem \ref{t2} $T(t)$ is local mean ergodic at $x$.\\
 	Moreover, we have $Ax=Az \notin N(A)$,  and by the following inclusion
 	$$ \sigma_{A}(Az)\subseteq \sigma_A(z) \subseteq \sigma_{A}(Az)\cup \{0\},$$
 	we have $\sigma_A(Az)=\{0\}$, if else we get $\sigma_A(Az)=\emptyset$, since $A$ has the SVEP then $Az=0$. Absurd, then we get that the local spectrum $\sigma_A(Ax)$ of $A$ at $Ax$ equal to $\{0\}$. Finally, we deduce that $\sigma_A(Ax)=\sigma_A(x)$, and by the seconds assertion of Corollary \ref{c1}, we obtain that $0$ is an essential singularity of the local resolvent function $\hat{x}_A$ of $A$ at $x$. 
 \end{ex}

  According to the Theorem \ref{t3}, we infer that if $x\in X_0$, then $x\in D(A)$. Thus, we have the following consequences, which will be deduced from Theorem \ref{t2}.
  
  \begin{coro}\label{c2}
 	Let $\{T(t)\}_{t\geqslant 0}$ be a $C_0$-semigroup  of bounded linear operators  on a  Banach space ${X}$ with $A$ be their infinitesimal generator. Let $x\in X_0$ and suppose that $A$ has the SVEP. If either of the following assertions hold:
 	\begin{enumerate}
 		\item  $\sigma_A(x)\neq \sigma_A(Ax)$, or 
 		\item   $0\notin \sigma_A(Ax)$, or 
 		\item  $r_{T(t)}(x) < 1$ \,  for all $t\geq 0$.
 	\end{enumerate}
 	Then $T(t)$ is a local mean ergodic at $x$.
 \end{coro}
\vspace{1mm}

  \begin{coro}\label{c3}
 	Let $\{T(t)\}_{t\geqslant 0}$ be a $C_0$-semigroup  of bounded linear operators on a   Banach space ${X}$ with $A$ be their infinitesimal generator. Suppose that  $A$ has the SVEP and $T(t)$ is a local mean ergodic at $x\in X$.  
If $0$ is a pole of the local resolvent function $\hat{x}_A$ of order $n$, then $n\leq 1$. 
 \end{coro}
 
 \vspace{1mm}

 \begin{coro}
	~  Let $\{T(t)\}_{t\geqslant 0}$ be a  $C_0$-semigroup  of bounded linear operators  on  a  Banach space ${X}$. Let $A$ be the infinitesimal generator of $T(t)$ and $R(., A)$ be their resolvent function. Assume that $\lim\limits_{t\to \infty} {\|T(t)x\|}/{t} =0$ for all $x\in X$. If one of the following assertions holds:
	
	$(i)$ \, $0\in \rho(A)$, or 
	
	$(ii)$\, $0$ is a pole of $R(\lambda, A)$.\\
	Then $T(t)$ is mean ergodic. 
\end{coro}
\vspace{1mm}

Next, we finish with the following general remark of this paper.
\begin{rem}\begin{enumerate}
		\item As mentioned above, a $C_0$-semigroup $\{T(t)\}_{t\geqslant 0}$ generated by $A$, such that $\lim\limits_{t\to \infty}{\|T(t)x\|}/{t} = 0$ for all $x\in X$ and $A$ has the SVEP, then $A$ is a quasi-nilpotent operator. Moreover, we can check that the Theorem \ref{thm2} is satisfied even if  $T(t)$ is mean ergodic.  
		\item \,  Let $A$ be the infinitesimal generator of the $C_0$-semigroup $\{T(t)\}_{t\geqslant 0} \, \subset  \mathcal{B}(X)$. \, If $A$ has the SVEP and $R(A)\cap H_0(A)=\{0\}$, then  $N(A)=H_0(A)$. In this case, we have $x\in R(A)$ if and only if $0\in\rho(A)$. Moreover, if $T(t)$ satisfies  $\lim\limits_{t\to \infty}{\|T(t)x\|}/{t}= 0$ for some  $x\in X$, then $T(t)$ is local mean ergodic at $x$. \\\\
		
	\end{enumerate}
\end{rem}



\begin{thebibliography}{99}	
 \bibitem{PA} \textsc{ P. Aiena, C.Trapani, S. Triolo, } \emph{ SVEP and local specral radius formula for unbounded operators, } Filomat 28:2 (2014), 263-273.
 
  \bibitem{PAiena} \textsc{ P. Aiena, } \emph{  Fredholm and Local Spectral Theory, With Applications to Multipliers, } Springer 2004.
 
  
 	\bibitem{Bur} \textsc{T. Berm{\'u}dez, M. Gonz{\'a}lez and A. Martin{\'o}n,} \emph{ On the Poles of the Local Resolvent,} Math. Nachr., 193 (1998), 19-26. 
 		\bibitem{E}
 	\textsc{I. Erdelyi and W. Shengwang,} \emph{ A Local Spectral Theory for Closed Operators,}
 	Cambridge University Press (London-New York-New Rochelle-Melbourne-Sydney), 1985.
		 \bibitem{HP}    \textsc{ E. Hille and R. S. Phillips,} \emph{ Functional analysis and semigroups,}  Amer. Math. Soc. Colloq. Publ., vol.
		  31, Amer. Math. Soc., Providence, R. I., 1957.
		  
\bibitem{Kido}    \textsc{K. Kido and W. Takahashi,} \emph{ Mean ergodic theorems for semigroups of linear operators,}
I. Math. Anal. Appl. 103 (1984), 387-394.
	\bibitem{KSZ}
\textsc{Y. Kozitsky, D. Shoikhet and J. Zem{\'a}nek,} \emph{ Power convergence of Abel averages,} Arch. Math. (Basel), 100 (2013), 539-549.
		\bibitem{Kr85}    \textsc{ U. Krengel,} \emph{ Ergodic Theorems,} Walter de Gruyter Studies in Mathematics 6, Walter de Gruyter, Berlin-New York, 1985.
		\bibitem{L2}    \textsc{ M. Lin,} \emph{ On the uniform ergodic theorem II,} Proc. Amer. Math. Soc., 46 (1974), 217-225.
			\bibitem{Su}\textsc{ M. Lin, D. Shoikhet and L. Suciu,} \emph{ Remarks on uniform ergodic theorems,}  Acta Sci. Math. (Szeged), 81 (2015), 251-283.
          	\bibitem{Mas}
\textsc{P. Masani,} \emph{  Ergodic theorems for locally integrable semigroups of continuous linear operators on a Banach space,} Advances in Math. 21 (1976), 202-228.

	\bibitem{Mb1}
\textsc{M. Mbekhta,} \emph{  G{\'e}n{\'e}ralisation de la d{\'e}composition de Kato aux op{\'e}rateurs paranormaux at spectraux,}
Glasgow Mathematical Journal 29 (1987), 159-75.
 	\bibitem{Mb}
 \textsc{M. Mbekhta,} \emph{  Th{\'e}orie spectrale locale et limite de nilpotents,} Proc. Amer. Math. Soc. 110 (1990) 621-631.
          \bibitem{Mul}    \textsc{V. M\"{u}ller,} \emph{Spectral Theory of Linear Operators and Spectral Systems in
	Banach Algebras} 2nd edition. Oper. Theory Advances and
Applications, vol. 139, 2007.

		\bibitem{Pa}    \textsc{ A. Pazy,} \emph{ Semigroups of Linear Operators and Applications to Partial Differential
			Equations},  Applied Mathematical Sciences, vol. 44, Springer-Verlag,  New York 1983.
		
			\bibitem{Sato} \textsc{R. Sato,} \emph{ On a mean ergodic theorem, } Proc. Amer. Math. Soc. 83 (1981), 563-564.
		
		\bibitem{Sh86}    \textsc{ S.Y. Shaw,} \emph{ Uniform ergodic theorems for locally integrable semigroups and pseudo-resolvents,}  Proc. Amer. Math. Soc., 98 (1986), 61-67.
		\bibitem{Sh-89}    \textsc{S.Y. Shaw,} \emph{ Mean ergodic theorems and linear functional equations,}   J. Functional Anal. 87 (1989), 428-441. 

 	  	\bibitem{taj}    \textsc{A. Tajmouati, M. Karmouni and F. Barki,} \emph{ Abel ergodic theorem for $C_0$-semigroups,} Adv. Oper. Theory, Vol. 5, No 4,  (2020), 1468-1479. 
 			 
\end{thebibliography}
\end{document}